\documentclass[a4paper, 12pt]{article}
\usepackage{amssymb}
\usepackage{amsmath}
\usepackage{amsthm}
\usepackage[dvipdfmx]{graphicx}
\usepackage{verbatim,enumerate}
\usepackage[active]{srcltx}
\usepackage{cite}
 \newtheorem{theorem}{Theorem}[section]
 \newtheorem*{theorem*}{Theorem}
 \newtheorem*{lemma*}{Lemma}
 \newtheorem{proposition}[theorem]{Proposition}
 \newtheorem{fact}[theorem]{Fact}
 \newtheorem{fact*}{Fact}
 \newtheorem{lemma}[theorem]{Lemma}
 
\theoremstyle{definition}
 \newtheorem{definition}[theorem]{Definition}
 \newtheorem{remark}[theorem]{Remark}
 \newtheorem*{remark*}{Remark}
 
 \newtheorem{example}[theorem]{Example}
 
\numberwithin{equation}{section}

\usepackage[usenames]{color}

\newcommand{\R}{\boldsymbol{R}}

\newcommand{\rank}{\operatorname{rank}}

\renewcommand{\phi}{\varphi}

\newcommand{\inner}[2]{\left\langle{#1},{#2}\right\rangle}
\newcommand{\spann}[1]{\left\langle{#1}\right\rangle_{\R}}

\newcommand{\ep}{\varepsilon}

\newcommand{\pmt}[1]{{\begin{pmatrix} #1  \end{pmatrix}}}

\newcommand{\mycomment}[1]{}
\title{\Large {\bf Normal form of swallowtail and its applications}}
\author{Kentaro Saji}
\date{\today}
\begin{document}
\maketitle
\begin{abstract}
We construct a form of swallowtail singularity in $\R^3$
which uses coordinate transformation on the source and
isometry on the target.
As an application, we classify configurations of
asymptotic curves and characteristic curves
near swallowtail.
\end{abstract}

\renewcommand{\thefootnote}{\fnsymbol{footnote}}
\section{Introduction}
Wave fronts and frontals are surfaces in $3$-space, and
they may have singularities.
They always have normal directions even along singularities.
Recently, there appeared several articles concerning on 
differential geometry of wave fronts and frontals
\cite{hhnsuy,hhnuy,hnuy,IO,MS,MSUY,nuy,OTflat,SUY}.
Surfaces which have only
cuspidal edges and swallowtails as singularities are 
the generic wave fronts in the Euclidean $3$-space.
Fundamental differential geometric invariants of cuspidal edge
is defined in \cite{SUY}.
It is further investigated in \cite{MS,MSUY,ist},
where the normal form of cuspidal edge
plays a important role.
The normal form of a singular point is a parametrization
using by coordinate transformation on the source
and isometric transformation on the target\cite{bw,WE}.
For the purpose of differential geometric investigation of
singularities,
it is not only convenient, but also
indispensable to studying higher order invariants.
Higher order invariants of cuspidal edges are
studied in \cite{nuy}, and in \cite{nuy},
moduli of isometric deformations of 
cuspidal edge is determined.
In this paper, we give a normal form of swallowtail,
and study relationships to previous investigation
of swallowtail.
As an application, we study geometric
foliations near swallowtail.

The precise definition of the swallowtail is given as follows:
The unit cotangent bundle $T^*_1\R^{3}$ of  $\R^{3}$ has the
canonical contact structure and can be identified with the unit
tangent bundle $T_1\R^{3}$. Let $\alpha$ denote the canonical
contact form on it. A map $i:M\to T_1\R^{3}$ is said to be 
{\it
isotropic\/} if the pull-back $i^*\alpha$ vanishes
identically, where $M$ is a $2$-manifold. 
If $i$ is an immersion, then we call the image of $\pi\circ i$
the {\it wave front set\/} of $i$,
where $\pi:T_1\R^{3}\to\R^{3}$ is the canonical projection and we
denote it by\/ $W(i)$. Moreover, $i$ is called the 
{\it Legendrian
lift\/} of $W(i)$. 
With this framework, we define the notion of
fronts as follows: A map-germ $f:(\R^2,0) \to (\R^{3},0)$ is
called a {\it frontal\/} 
if there exists a unit vector field 
(called {\it unit normal of\/} $f$)
$\nu$ of $\R^{3}$ along $f$
such that
$L=(f,\nu):(\R^2,0)\to (T_1\R^{3},0)$ is
an isotropic map by an identification 
$T_1\R^3 = \R^3 \times S^2$, where $S^2$ is 
the unit sphere in $\R^3$ (cf. \cite{AGV}, see also \cite{krsuy}).
A frontal $f$ is a {\it front\/} if the above $L$ can be taken as
an immersion.
A point $q\in (\R^2, 0)$ is a singular point if $f$ is not an
immersion at $q$.
A map $f:M\to N$ 
between $2$-dimensional manifold $M$ and
$3$-dimensional manifold $N$ is called 
a frontal (respectively, a front)
if for any $p\in M$, the map-germ $f$ at $p$
is a frontal (respectively, a front).
A singular point $p$ of a map $f$ is called a 
{\it cuspidal edge\/}
if the map-germ $f$ at $p$ is $\mathcal{A}$-equivalent to
$(u,v)\mapsto(u,v^2,v^3)$ at $0$,
and a singular point $p$ is called a {\it swallowtail\/}
if the map-germ $f$ at $p$ is $\mathcal{A}$-equivalent to
$(u,v)\mapsto(u,4v^3+2uv,3v^4+uv^2)$ at $0$,
 (Two map-germs
$f_1,f_2:(\R^n,0)\to(\R^m,0)$ are $\mathcal{A}$-
{\it equivalent}\/ if there exist diffeomorphisms
$S:(\R^n,0)\to(\R^n,0)$ and $T:(\R^m,0)\to(\R^m,0)$ such
that $ f_2\circ S=T\circ f_1 $.) Therefore if the singular point $p$
of $f$ is a swallowtail, then $f$ at $p$ is a front.
\section{Singular points of $k$-th kind}

Let $f:(\R^2,0) \to (\R^{3},0)$ be a frontal and $\nu$ its
unit normal.
Let $\lambda$ be a function which is a non-zero functional multiplication
of the function
$$
\det(f_u,f_v,\nu)
$$
for some coordinate system $(u,v)$, and 
$(~)_u=\partial/\partial u$,
$(~)_v=\partial/\partial v$.
We call such function {\it singularity identifier}.
A singular point $p$ of $f$ is called {\it non-degenerate\/}
if $d\lambda(p)\ne0$.
Let $0$ be a non-degenerate singular point of $f$.
Then the set of singular points $S(f)$ is a regular curve,
we take a parameterization $\gamma(t)$ $(\gamma(0)=0)$ of it.
We set $\hat\gamma=f\circ\gamma$ and call $\hat\gamma$ 
{\it singular locus}.
One can show that
there exists a vector field $\eta$
such that if $p\in S(f)$, then
$$
\ker df_p=\spann{\eta_p}.
$$
We call $\eta$ the {\it null vector field}.
On $S(f)$, $\eta$ can be parameterized by
the parameter $t$ of $\gamma$.
We denote by $\eta(t)$ the null vector field along $\gamma$.
We set 
\begin{equation}\label{eq:criphi}
\phi(t)=\det\left(\dfrac{d\gamma}{dt}(t),\eta(t)\right).
\end{equation}
\begin{definition}
A non-degenerate singular point $0$ is 
the {\it first kind\/} if
$\phi(0)\ne0$.
A non-degenerate singular point $0$ is 
the {\it $k$-th kind\/} ($k\geq2$) if
$$
\dfrac{d\phi}{dt}(0)=\cdots
=
\dfrac{d^{k-2}\phi}{dt^{k-2}}(0)=0,\quad
\dfrac{d^{k-1}\phi}{dt^{k-1}}(0)\ne0.
$$
\end{definition}
The definition does not depend on the choice
of the parameterization of $\gamma$ and
choice of $\eta$.
We remark that if $f$ is a front,
then the singular point of the first kind
is the cuspidal edge, and
the singular point of the second kind
is the swallowtail \cite{krsuy}.
We can rephrase the definition of the $k$-th kind singularity
as follows.
Let $0$ be a non-degenerate singular point of $f$.
Then there exists a vector field $\tilde\eta$
such that if $p\in S(f)$ then
$
\ker df_{p}=\spann{\eta_p}.
$
We call $\tilde\eta$ the {\it extended null vector field}.

\begin{lemma}\label{lem:kthlambda}
Let\/ $0$ be a non-degenerate singular 
point\/ $0$ of a frontal\/
$f:(\R^2,0)\to(\R^3,0)$,
and let\/ $\lambda$ be a singularity identifier.
Then the followings are equivalent.
\begin{enumerate}
\item \label{item:kth1} $0$ is a singular point of the\/ $k$-th kind,
\item \label{item:kth2}
$
\eta\lambda=\cdots=\eta^{k-1}\lambda(0)=0,
\eta^{k}\lambda(0)\ne0,
$
where\/
$\eta$
is a null vector field, and\/
$\eta^i$ stands for the\/ $i$ times directional derivative by\/ $\eta$.
\end{enumerate}
\end{lemma}
Firstly we show that the condition
\ref{item:kth2} does not depend on the choices
of $\eta$ and $\lambda$.
It is obvious that for the choice of 
$\lambda$,
we shoe it for the choice of
$\tilde\eta$.
We show the following lemma.
\begin{lemma}
Let\/ $\lambda:(\R^2,0)\to(\R,0)$ be 
a function satisfying\/
$d\lambda(0)\ne0$, and let\/
$\eta$ be a vector field.
Let\/ $\overline\eta$ be another vector field
satisfying \/
$\overline\eta=h\tilde\eta$, 
where\/ $h$ is a function\/ $h(0)\ne0$,
on\/
$\lambda^{-1}(0)$.
Then, if
\begin{equation}\label{eq:katei}
\tilde\eta\lambda=\cdots=\tilde\eta^{k-1}\lambda(0)=0,\quad
\tilde\eta^{k}\lambda(0)\ne0\quad(k\geq1)
\end{equation}
hold, then
\begin{equation}\label{eq:keturon}
\overline\eta\lambda=\cdots=\overline\eta^{k-1}\lambda(0)=0,\quad
\overline\eta^{k}\lambda(0)\ne0\quad(k\geq1)
\end{equation}
hold.
\end{lemma}
\begin{proof}
Without loss of generality, we can assume
the coordinate system $(u,v)$ satisfies
$\tilde\eta=\partial v$.
By $\eta\lambda=0$ and $d\lambda(0)\ne0$, we have
$\lambda_u\ne0$.
Thus by the implicit function theorem,
there exists a function $a(v)$ such that
$$
\lambda(a(v),v)=0.
$$
Thus $\lambda$ is proportional to
$a(v)-u$, and without loss of generality,
we can assume $\lambda=a(v)-u$.
By the assumption \eqref{eq:katei},
$a(0)=\cdots=a^{(k-1)}(0)=0$, $a^{(k)}(0)\ne0$ holds.
We show \eqref{eq:keturon}. 
Since \eqref{eq:keturon} does not depend on the
non-zero functional multiplication 
of $\overline\eta$,
we may assume
$$
\overline\eta=b(u,v)\partial u+\partial v.
$$
We show that
\begin{equation}\label{eq:etal}
\overline\eta^l\lambda
=
b(u,v)h_0(u,v)
+
\sum_{j=1}^{l-1}\dfrac{\partial^j b}{\partial v^j}(u,v)h_j(u,v)
+a^{(l)}\quad (h_0,\ldots,h_{l-1}\text{ are functions}),
\end{equation}
by the induction.
When $l=1$, since
$\overline\eta\lambda=b\lambda_u+\lambda_v
=-b+a'$, \eqref{eq:etal} is true.
We assume that \eqref{eq:etal} for $l=i$.
Since
\begin{align*}
\overline\eta^{i+1}\lambda
&=
\overline\eta(\overline\eta^{i}\lambda)\\
&=
b\Big(
bh_0
+
b_vh_1
+
\cdots
+
b_{v^{i-1}}h_{i-1}+a^{(i)}\Big)_u\\
&\hspace{15mm}
+
b_vh_0+b(h_0)_v
+
b_{vv}h_1+b_v(h_1)_v
+
\cdots
+
b_{v^i}h_{i-1}
b_{v^{i-1}}(h_{i-1})_v+a^{(i+1)},
\end{align*}
\eqref{eq:etal} is true for $l=i+1$.
\end{proof}
\begin{proof}[Proof of Lemma {\rm \ref{lem:kthlambda}}]
We show the case $k\geq2$, since it is clear when $k=1$.
Since $0$ is non-degenerate, we can take a coordinate
system $(u,v)$ satisfying $S(f)=\{v=0\}$.
By the non-degeneracy, we may assume $\lambda=v$.
Furthermore, we can take
$\eta(u)=\partial_u+\ep(u)\partial_v$ as a null vector field.
Then $\phi(u)=\ep(u)$, \ref{item:kth1} is equivalent to
$\ep'(0)=\cdots=\ep^{(k-2)}(0)=0$ and
$
\ep^{(k-1)}(0)\ne0.
$
On the other hand, since
$\lambda=v$, it holds that
$\eta\lambda=\ep(u)$.
Then this depends only on $u$,
$\eta^2\lambda=\ep'(u)$ holds, and
$\eta^l\lambda=\ep^{(l-1)}(u)$ holds.
Thus \ref{item:kth2} is equivalent to
$\ep'(0)=\cdots=\ep^{(k-2)}(0)=0$ and
$\ep^{(k-1)}(0)\ne0$.
Hence we have the equivalency of
\ref{item:kth1} and \ref{item:kth2}.
\end{proof}

\section{Normal form of singular point of the second kind}
In this section,
we construct a normal form of the
singular point of the second kind which includes
swallowtail.
Furthermore, we study the relationships to
the known invariants of swallowtail.
Throughout this section,
let $f:(\R^2,0) \to (\R^{3},0)$ be a frontal and $\nu$ its
unit normal, and
let $0$ be a singular point of the second kind.
\subsection{Normal form of singular point of the second kind}
We take coordinate transformation on the source
and isometric transformation on the target,
we detect the normal form of singular points
of second kind.
By the non-degeneracy, $\rank df_0=1$ follows,
and by rotating coordinate system on the target,
we may assume that $f_u(0,0)=(a,0,0)$, $a>0$.
By changing coordinate system on the source,
we may assume $f$ has the form
$$
f(u,v)=(u,f_2(u,v),f_3(u,v)),\quad
f_u(0,0)=(1,0,0).
$$
Since the Jacobian matrix of $f$ is
$$
\pmt{
1&0\\
(f_2)_u(u,v)&(f_2)_v(u,v)\\
(f_3)_u(u,v)&(f_3)_v(u,v)},
$$
$S(f)=\{(f_2)_v=(f_3)_v=0\}$.
Thus we can take the null vector field
$\eta=\partial v$.
Since $0$ is non-degenerate,
$S(f)$ can be parametrized by
$(s(v),v)$ near $0$.
Moreover, $0$ is a singular point of the second kind,
$s(0)=s'(0)=0$, and $s''(0)\ne0$ hold.
We may assume $s''(0)>0$ by 
changing $(u,v)\mapsto(-u,-v)$ if necessary.
Thus there exists a function $\tilde s$ such that
$$
s(v)=\dfrac{v^2\tilde s(v)}{2},\quad \tilde s(0)>0.
$$
Setting
$
t(v)=\sqrt{s(v)},
$
we have
$$
s(v)=\dfrac{(vt(v))^2}{2},\quad t(0)\ne0.
$$
We take the diffeomorphism $\phi$ on the source
defined by
$$
\phi(u,v)=(u,vt(v)),
$$
and consider $(\tilde u,\tilde v)=\phi(u,v)$
as the new coordinate system.
Since
$$
\phi(s(v),v)
=
(s(v),vt(v))
=
\left(\dfrac{(vt(v))^2}{2},vt(v)\right)
=
(\tilde v^2/2,\tilde v),
$$
$S(f)=\{(\tilde v^2/2,\tilde v)\}$ holds.
Furthermore, the first component of $f(\tilde u,\tilde v)$
is $\tilde u$, we see that
$\partial \tilde v$ is a null vector field.
Now we may assume that $f$ has the form
\begin{itemize}
\item $f(u,v)=(u,f_2(u,v),f_3(u,v))$,
\item $\partial v$ is a null vector field,
\item $S(f)=\{(v^2/2,v)\}$.
\end{itemize}
Since $(f_2)_v$ and $(f_3)_v$ vanish on
$S(f)=\{u=v^2/2\}$,
there exist functions $g_1,h_1$ such that
$$
(f_2)_v(u,v)=(v^2/2-u)g_1(u,v),\quad
(f_3)_v(u,v)=(v^2/2-u)h_1(u,v).
$$
By the non-degeneracy, 
$(g_1(0,0),h_1(0,0))\ne(0,0)$.
Since
$$
f_2(u,v)=\int(v^2/2-u)g_1(u,v)\,dv,\quad
f_3(u,v)=\int(v^2/2-u)h_1(u,v)\,dv,
$$
taking the partial integration,
\begin{equation}\label{eq:int}
\begin{array}{rcl}
\displaystyle
f_2(u,v)
&=&\displaystyle
\int(v^2/2-u)g_1(u,v)\,dv\\
&=&\displaystyle
(v^2/2-u)g_2(u,v)-\int v g_2(u,v)\,dv\\
&=&\displaystyle
(v^2/2-u)g_2(u,v)-v g_3(u,v)+\int g_3(u,v)\,dv\\
&=&\displaystyle
(v^2/2-u)g_2(u,v)-v g_3(u,v)+g_4(u,v)
\end{array}
\end{equation}
holds, where
$$
g_i(u,v)=\dfrac{\partial g_{i+1}}{\partial v}(u,v).
$$
Similarly, we have
\begin{equation}\label{eq:int2}
f_3(u,v)
=(v^2/2-u)h_2(u,v)-v h_3(u,v)+h_4(u,v),\quad
h_i(u,v)=\dfrac{\partial h_{i+1}}{\partial v}(u,v).
\end{equation}
Since $(g_1(0,0),h_1(0,0))\ne(0,0)$,
\begin{equation}\label{eq:nu2}
\nu=\dfrac{\nu_2}{|\nu_2|},\quad
\Big(\nu_2
=
(h_1f_{2u}-g_1f_{3u},-h_1,g_1)\Big)
\end{equation}
gives a unit normal vector for $f$,
because of 
\begin{align}
f_u(u,v)&=\big(1,(f_2)_u,(f_3)_u\big)\nonumber\\
&=\big(1,
-g_2+(v^2/2-u)g_{2u}-vg_{3u}+g_{4u},\nonumber\\
&\hspace{40mm}
-h_2+(v^2/2-u)h_{2u}-vh_{3u}+h_{4u}\big),
\label{eq:f2u}\\
f_v(u,v)
&=
\big(0,(v^2/2-u)g_1,(v^2/2-u)h_1\big).
\end{align}
Since $f_v(0,0)=0$, $f$ is a front
if and only if $\nu_v(0,0)\ne0$,
and it is equivalent to that
$\nu_2$ and $\nu_{2v}$ are linearly independent.
Since
\begin{align*}
\nu_{2v}
=&
(h_0f_{2u}+h_1f_{2uv}-g_0f_{3u}-g_1f_{3uv},-h_0,g_0)\\
=&
\Big(h_0f_{2u}+h_1(-g_1(u,v)+(v^2/2-u)g_{1u}(u,v))\\
&\hspace{10mm}-g_0f_{3u}-g_1(-h_1(u,v)+(v^2/2-u)h_{1u}(u,v)),-h_0,g_0
\Big)\\
=&
(h_0f_{2u}+(v^2/2-u)h_1g_{1u}(u,v)
-g_0f_{3u}-(v^2/2-u)g_1h_{1u}(u,v),-h_0,g_0)
\end{align*}
and
\begin{align*}
&
\rank\pmt{
h_1f_{2u}-g_1f_{3u}&-h_1&g_1\\
h_0f_{2u}+(v^2/2-u)h_1g_{1u}(u,v)
-g_0f_{3u}-(v^2/2-u)g_1h_{1u}(u,v)&-h_0&g_0}(0)\\
=&
\rank\pmt{
0&-h_1&g_1\\
(v^2/2-u)h_1g_{1u}(u,v)
-(v^2/2-u)g_1h_{1u}(u,v)&-h_0&g_0}(0)
\\
=&
\rank\pmt{
0&-h_1&g_1\\
0&-h_0&g_0}(0),
\end{align*}
it is equivalent to
$$
\det\pmt
{
g_1(0)&g_0(0)\\
h_1(0)&h_0(0)
}\ne0.
$$
By rotating coordinate system on the target
around the axis which contains $(1,0,0)$,
we may assume 
$
\nu_2/|\nu_2|
=
(0,0,1)
$, namely, 
$g_1(0,0)=g_{4vvv}(0,0)>0$ and
$h_1(0,0)=h_{4vvv}(0,0)=0$.
Moreover, by $f(0,0)=(0,g_4(0,0),h_4(0,0))$,
we have $g_4(0,0)=0$, $h_4(0,0)=0$,
and by
$f_u(0,0)=(1,0,0)$ and
\eqref{eq:f2u},
\begin{align*}
f_{2u}(0,0)=&-g_2(0,0)+g_{4u}(0,0)=-g_{4vv}(0,0)+g_{4u}(0,0)=0,\\
f_{3u}(0,0)=&-h_2(0,0)+h_{4u}(0,0)=-h_{4vv}(0,0)+h_{4u}(0,0)=0.
\end{align*}
Summarizing up the above arguments,
we have the following proposition.
\begin{proposition}\label{swnormal}
For any function\/ $g$ and\/ $h$ satisfying\/
$g_{vvv}(0,0)>0$,
$g(0,0)=h(0,0)=0$,
$g_u(0,0)-g_{vv}(0,0)=0$,
$h_u(0,0)-h_{vv}(0,0)=0$ and\/
$h_{vvv}(0,0)=0$,
\begin{equation}\label{eq:swnormal}
\begin{array}{rl}
f(u,v)=&\Bigg(
u,\ 
\left(\dfrac{v^2}{2}-u\right)
g_{vv}(u,v)-vg_v(u,v)+g(u,v),\\
&\hspace{15mm}
\left(\dfrac{v^2}{2}-u\right)
h_{vv}(u,v)-vh_v(u,v)+h(u,v)\Bigg)
\end{array}
\end{equation}
is a frontal satisfying
that\/
$0$ is a singular point of the second kind,
and\/
$f_u(0,0)=(1,0,0)$, $\eta=\partial_v$, $S(f)=\{v^2/2-u=0\}$.
Moreover, if
$$
h_{vvvv}(0,0)\ne0,
$$
then\/ $0$ is a swallowtail.
Conversely, for any 
singular point of the second kind\/ $p$ of
a frontal\/ $f:U\to\R^3$, there exists
a coordinate system\/ $(u,v)$ on\/ $U$, 
and an orientation preserving isometry\/
$\Phi$ on\/ $\R^3$ such that\/
$\Phi\circ f(u,v)$ can be written in the form\/
\eqref{eq:swnormal}.
\end{proposition}
\begin{remark}
Conditions $g(0,0)=h(0,0)=0$,
$g_u(0,0)-g_{vv}(0,0)=0$,
$h_u(0,0)-h_{vv}(0,0)=0$,
$g_{vvv}(0,0)>0$,
$h_{vvv}(0,0)=0$
are just for the reducing coefficients.
If one want to obtain 
a second kind singular point (respectively swallowtail),
taking
$g$ and $h$ satisfying that
$$
(g_{vvv}(0,0),h_{vvv}(0,0))\ne(0,0),\quad
\left(\text{respectively,\ }\det
\pmt{
g_{vvv}(0,0)&h_{vvv}(0,0)\\
g_{vvvv}(0,0)&h_{vvvv}(0,0)}
\ne0\right),
$$
and forming \eqref{eq:swnormal} is enough.
\end{remark}
We remark that 
another different normal form of swallowtail is obtained
in \cite{fukui} by a different view point.
\begin{example}
Let us set
$$
g=v^3/6,\quad 
h=v^k/k!,\quad(k=4,5,6),
$$
and consider \eqref{eq:swnormal}.
Then the figure of $f$
can be drawn as in Figure \ref{fig:swexs}.
\begin{figure}[!ht]
\centering
\includegraphics[width=.25\linewidth]{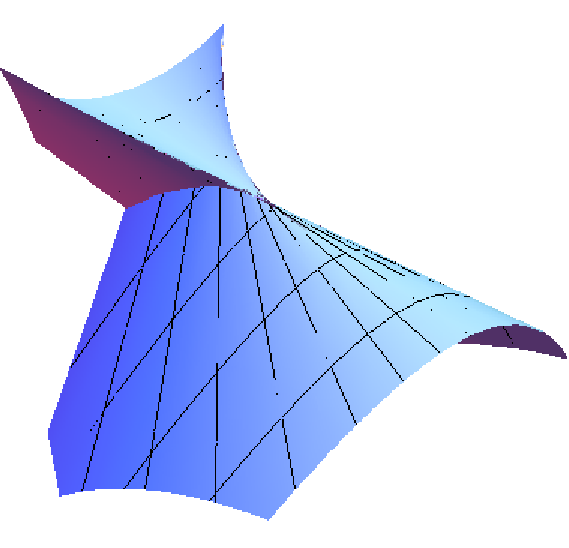}
\hspace{1mm}
\includegraphics[width=.25\linewidth]{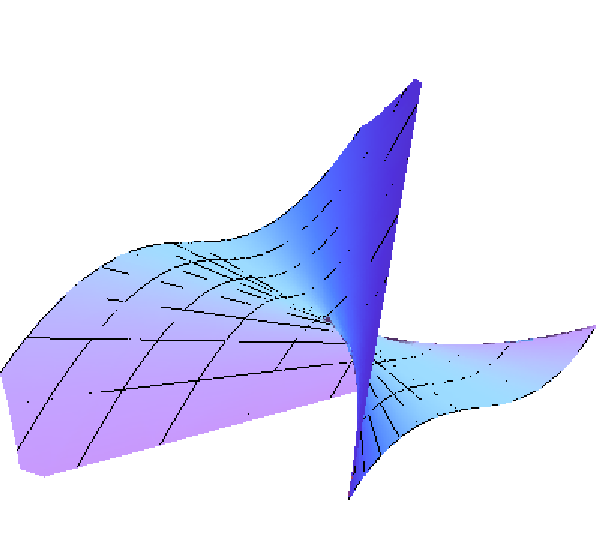}
\hspace{1mm}
\includegraphics[width=.25\linewidth]{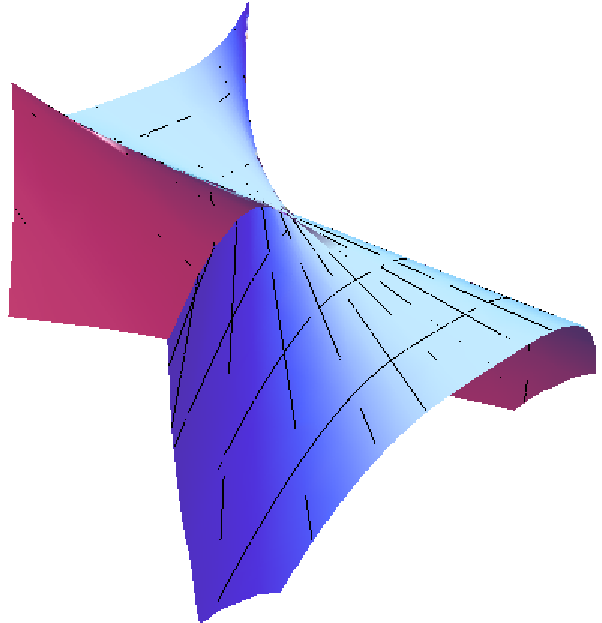}
\caption{$f$ of $k=4,5,6$}
\label{fig:swexs}
\end{figure}
\end{example}
\begin{remark}
By the above construction,
we can obtain the normal forms for 
singular points of $k$-th kind by the same
mannar.
For functions $g,h$,
\begin{equation}\label{eq:normalbig}
\begin{array}{l}
\Bigg(u,
\displaystyle
\left(\dfrac{v^k}{k!}-u\right)g^{(k)}+
\sum_{i=1}^k(-1)^i\dfrac{v^{k-i}}{(k-i)!}g^{(k-i)}(u,v),\\
\hspace{30mm}
\displaystyle
\left(\dfrac{v^k}{k!}-u\right)h^{(k)}+
\sum_{i=1}^k(-1)^i\dfrac{v^{k-i}}{(k-i)!}h^{(k-i)}(u,v)
\Bigg)
\end{array}
\end{equation}
at $0$ is a $k$-th kind of singular point 
if
$(g^{(k)},h^{(k)})(0,0)\ne(0,0).$
Moreover, if
$$
\det\pmt
{
g^{(k+1)}(0)&g^{(k+2)}(0)
\\
h^{(k+1)}(0)&g^{(k+2)}(0)
}\ne0,
$$
then it is a front.
Here, $g'=\partial g/\partial v=g^{(1)}$,
and
$g^{(i)}=\partial g^{(i-1)}/\partial v$, for example.
\end{remark}
\begin{example}\label{eq:normalbut}
Let us set
$g=v^4/4!$ and $h=u^2/2+v^5/5!$. 
Then the surface obtained by \eqref{eq:normalbig}
is $(u,v)\mapsto
(u, -u v + v^4/24, (-15 u^2 + 15 u v^2 - v^5)/30)$,
and it 
can be drawn in Figure \ref{fig:normalbig}.
This singularity is called cuspidal butterfly.
\begin{figure}[!ht]
\centering
\includegraphics[width=.25\linewidth]{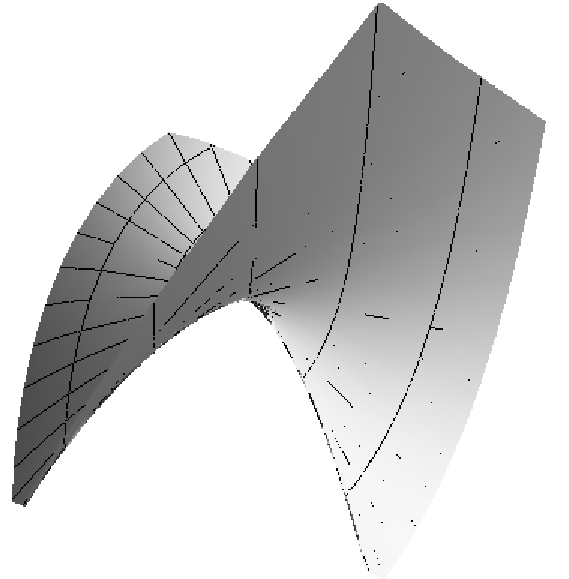}
\caption{The surface of example \ref{eq:normalbut}.}
\label{fig:normalbig}
\end{figure}
\end{example}

\subsection{Normal form that the singular set is the $u$-axis}
\label{sec:singuaxis}
The singular set $S(f)$ of $f$ in the form \eqref{eq:swnormal}
is a parabola and the null vector field on $S(f)$ is
constantly $\partial_v$.
On the other hand,
sometimes we want to have a form
satisfying that the singular set is the $u$-axis,
although the null vector field is not constant.
For that purpose,
take $f$ as in \eqref{eq:swnormal},
and set
$\tilde f(\tilde u,\tilde v)=
f(\tilde u + \tilde v^2/2, \tilde v)$.
Then 
$F(x,y)=-\tilde f(-y,x)$
is a frontal, and $0$ is a singular point of
the second kind satisfying
$\eta=\partial_u+u\partial_v$ and $S(f)=\{v=0\}$.

\subsection{Forms in the low degrees}
In the Proposition \ref{swnormal},
we have the normal forms in the low degrees
in the following manner.
In the form \eqref{eq:swnormal},
we set
$$
g(u,v)=g_5(u,v)+g_6(u,v),\quad
h(u,v)=h_5(u,v)+h_6(u,v),
$$
where $g_6,h_6$ satisfy $j^5g_6(0,0)=j^5h_6(0,0)=0$,
and 
$g_5,h_5$ are
$$
g_5(u,v)
=
\sum_{i+j=1}^5\dfrac{a_{ij}}{i!j!}u^iv^j,\quad
h_5(u,v)
=
\sum_{i+j=1}^5\dfrac{b_{ij}}{i!j!}u^iv^j,
$$
where 
$
a_{ij},b_{ij}\in\R$
and $a_{02}=a_{10}$, $b_{02}=b_{10}$, $a_{03}\ne0$, $b_{03}=0$.
Then 
$f$ has the form
\begin{equation}\label{eq:swnormalkeisu}
\begin{array}{l}
\displaystyle
\Bigg(u,\ 
\dfrac{-2 a_{12} + a_{20}}{2} u^2 
+ \dfrac{-3 a_{22} + a_{30}}{6} u^3 
+ \dfrac{-4 a_{32} + a_{40}}{24} u^4
- a_{03} u v - a_{13} u^2 v
- \dfrac{a_{23}}{2} u^3 v  \\[4mm]
\displaystyle
\hspace{15mm} 
- \dfrac{a_{04}}{2} u v^2 
- \dfrac{a_{14}}{2} u^2 v^2
+ \dfrac{a_{03}}{6} v^3 
+ \dfrac{-a_{05} + a_{13}}{6} u v^3 
+ \dfrac{a_{04}}{8} v^4+G(u,v),\\[4mm]
\displaystyle
\hspace{30mm}
\dfrac{-2 b_{12} + b_{20}}{2} u^2 
+\dfrac{-3 b_{22} + b_{30}}{6} u^3 
+ \dfrac{-4 b_{32} + b_{40}}{24} u^4
- b_{13} u^2 v 
- \dfrac{b_{23}}{2} u^3 v   \\[4mm]
\displaystyle
\hspace{30mm}
- \dfrac{b_{04}}{2} u v^2 
- \dfrac{b_{14}}{2} u^2 v^2
+ \dfrac{-b_{05} + b_{13}}{6} u v^3 
+ \dfrac{b_{04}}{8} v^4+H(u,v)
\Bigg),
\end{array}
\end{equation}
where, $G,H$ are functions their $4$-jet vanishes:
$j^4G(0,0)=j^4H(0,0)=0$,
and
$$
\begin{array}{l}
G(u,v)=\dfrac{1}{2}(g_5)_{vv}(u,v)v^2
+\left(\dfrac{v^2}{2}-u\right)(g_6)_{vv}(u,v)
-v(g_6)_v(u,v)+g_6(u,v)\\[4mm]
H(u,v)=\dfrac{1}{2}(h_5)_{vv}(u,v)v^2
+\left(\dfrac{v^2}{2}-u\right)(h_6)_{vv}(u,v)
-v(h_6)_v(u,v)+h_6(u,v).
\end{array}
$$
\subsection{Invariants}
In \cite{MSUY}, several invariants of 
singular points of the second kind
are introduced.
We take a parametrization $\gamma(t)$ of $S(f)$ and
assume $\gamma(0)=0$.
We set $\hat\gamma=f\circ\gamma$ as above.
The {\it limiting normal curvature\/}
$\kappa_\nu$ of $f$ at $0$ is defined by 
$$
\kappa_\nu(0)=
\lim_{t\to 0}\dfrac{\inner{\hat\gamma''(t)}{\nu(\gamma(t))}}
{|\hat\gamma(t)|^2}
$$
with respect to the unit normal vector $\nu$
(cf. \cite[(2.2)]{MSUY}),
where $\hat\gamma$ is the singular locus.
The {\it normalized cuspidal curvature\/} $\mu_c$
is defined by 
$$
\mu_c=\left.\dfrac{-|f_u|^3\inner{f_{uv}}{\nu_v}}
{|f_{uv}\times f_u|^2}\right|_{(u,v)=(0,0)}
$$
(cf. \cite[(4.6)]{MSUY}), where
$(u,v)$ is a coordinate system satisfying
$\ker df_0=\spann{\partial v}$.
The limiting normal curvature and
normalized cuspidal curvature relate
the boundedness of the Gaussian and mean
curvature near singular points of the second kind
\cite[Propositions 4.2, 4.3, Theorem 4.4]{MSUY}.
The {\it limiting singular curvature\/} $\tau_s$
is defined by the limit of singular curvature \cite{SUY}
and it is computed by
$$
\tau_s=
\left.\dfrac{
\det(\hat\gamma'',\hat\gamma''',\nu(\gamma))}{|\hat\gamma''|^{5/2}}
\right|_{t=0}
$$
(cf. \cite[p. 272, Proposition 4.9]{MSUY}).
The limiting singular curvature
measures the wideness of the cusp
of the singular points of the second kind.

We assume that $f$ is a singular points of the second kind
given in the form \eqref{eq:swnormalkeisu}.
Then
\begin{equation}\label{eq:invkeisan}
\kappa_\nu(u)
=
-2 b_{12} + b_{20},\quad
\mu_c
=
\dfrac{b_{04}}{{a_{03}}^2},\quad
\tau_s=2a_{03},
\end{equation}
where 
$\nu$ is the unit normal vector satisfying $\nu(0,0)=(0,0,1)$.

\section{Geometric foliations near swallowtail}
In this section, as an application of the
normal form of swallowtail,
we study geometric foliations near swallowtail
defined by binary differential equations.
\subsection{Binary differential equations}
Let $U\subset \R^2$ be an open set and $(u,v)$ a coordinate
system on $U$.
Consider a $2$-tensor
\begin{equation}\label{eq:bdeabc}
\omega=p\,du^2+2q\,dudv+r\,dv^2
\end{equation}
where $p,q,r$ are functions on $U$.
If a vector field $X=x_1\partial_u+x_2\partial_v$ 
satisfies 
$$\omega(X,X)=px_1^2+2qx_1x_2+rx_2^2=0,$$
then 
direction of $X$ is called the direction
of $\omega=0$, and
the integral curves of $X$ is
called the solutions of 
$\omega=0$.
We call $\omega=0$ a
{\it binary differential equation\/} (BDE).
We set $\delta=q^2-pr$.
Then 
$\omega=0$ defines two linearly independent
directions on $\{\delta>0\}\subset U$,
and
it defines one direction on 
$\{\delta=0\}$, and
it defines no direction on 
$\{\delta<0\}$.

\begin{definition}
Two BDEs $\omega_1=0$, $\omega_2=0$
are {\it equivalent\/}
if there exist diffeomorphism
$\Phi:(\R^2,0)\to(\R^2,0)$,
and a function $\rho:(\R^2,0)\to\R$ $(\rho(0)\ne0)$
such that
$$\rho\ \Phi^*\omega_1=\omega_2.$$
\end{definition}
We identify two BDEs if they are equivalent.
If a $2$-tensor $\omega$ as in \eqref{eq:bdeabc} satisfies
$\delta(0)>0$, then the BDE
$\omega=0$ is equivalent to 
$dx^2-dy^2=0$.
We consider here the case
$r(0)\ne0$, $p_u(0)=0$, $p_v(0)\ne0$
following \cite{bt}, 
since only this case is needed for our consideration.
See \cite{btnl,bt,d,d2} for general study of BDEs.
Dividing $\omega$ by $r$,
and putting
$
\tilde p=p/r$, $\tilde q=q/r
$,
we consider 
\begin{equation}\label{eq:bde1}
\begin{array}{rl}
\tilde \omega&=
\tilde p \,du^2+2\tilde q \,dudv+dv^2,\\
\tilde p&=p_{01}v
+\dfrac{p_{20}}{2}u^2
+p_{11}uv
+\dfrac{p_{02}}{2}v^2+O(3),\quad
p_{01}\ne0\\
\tilde q&=q_{10}u+q_{01}v
+\dfrac{q_{20}}{2}u^2
+q_{11}uv
+\dfrac{q_{02}}{2}v^2+O(3),\\
\end{array}
\end{equation}
where $O(r)$ stands for the terms whose degrees are
greater than or equal to $r$.
We may assume $p_{01}>0$ without loss of
generality.
Considering the coordinate change
$$
u=
-\sqrt{p_{01}}\,U,\quad
v=V - \dfrac{q_{10}}{2p_{01}} U^2 
+ \dfrac{q_{01}}{\sqrt{p_{01}}} UV,
$$
and dividing by the coefficient of
$dv^2$, we see 
$\tilde\omega=0$ is equivalent to
$(A/C)\,dU^2+2(B/C)\,dUdV+dV^2=0$, where
\begin{align*}
A&=
V
+
\dfrac{p_{20}- 2q_{10}^2-p_{01}q_{10}}{2 p_{01}^2}U^2
-
\dfrac{p_{11}-2q_{10}q_{01}-p_{01}q_{01}}{p_{01}\sqrt{p_{01}}}
U V
+
\dfrac{p_{02}-2 q_{01}^2}{2 p_{01}}V^2
+O(3)\\
2B&=\dfrac{q_{10}q_{01}-q_{20}}{p_{01}\sqrt{p_{01}}}U^2
-
2\dfrac{q_{01}^2-q_{11}}{p_{01}}U V
-
\dfrac{q_{02}}{\sqrt{p_{01}}}V^2
+O(3)\\
C&=
1+\dfrac{2 q_{01} U}{\sqrt{p_{01}}}+\dfrac{q_{01}^2 U^2}{p_{01}}
+O(3)
\end{align*}
and it is equal to
$A'\,dU^2+2B'\,dUdV+dV^2=0$, where
\begin{align*}
A'&=
V
+
\dfrac{p_{20}- 2q_{10}^2-q_{10}p_{01}}{p_{01}^2}
\dfrac{U^2}{2}
-
\dfrac{p_{11}-2 q_{01} q_{10}+q_{01} p_{01}}{p_{01}\sqrt{p_{01}}}
U V
+
\dfrac{p_{02}-2 q_{01}^2}{p_{01}}\dfrac{V^2}{2}
+O(3)\\
2B'&=
\dfrac{q_{01} q_{10}-q_{20}}{p_{01}\sqrt{p_{01}}}U^2
-
2\dfrac{q_{01}^2-q_{11}}{p_{01}}U V
-
\dfrac{q_{02}}{\sqrt{p_{01}}}V^2
+O(3)
\end{align*}
Now we consider a BDE 
$\overline\omega=\overline p\,du^2+2\overline q\,dudv
+dv^2=0$, where
$$
\overline p
=
v+\dfrac{\overline p_{20}}{2}u^2
+\overline p_{11}uv+\dfrac{\overline p_{02}}{2}v^2+O(3),\quad
\overline q=\dfrac{\overline q_{20}}{2}u^2
+\overline q_{11}uv+\dfrac{\overline q_{02}}{2}v^2+O(3),\quad
p_{20}\ne0.$$
Consider a coordinate transformation
$$
u=
U+\dfrac{x_{20}}{2} U^2 + 
x_{11} UV,\quad
v=V + 
\dfrac{x_{30}}{6} U^3 
+ \dfrac{x_{21}}{2} U^2 V 
+ \dfrac{x_{12}}{2} UV^2,
$$
where
$$
x_{20}=-\dfrac{\overline p_{11}}{2},\ 
x_{11}=-\dfrac{\overline p_{02}}{4},\ 
y_{30}=-\overline q_{20},\ 
y_{21}=\dfrac{-4 \overline q_{11} 
+ \overline p_{02}}{4},\ 
y_{12}=-\overline p_{02},
$$
and dividing by the coefficient of
$dv^2$,
we see $\overline \omega=0$ is equivalent to
$(P/R)\,dU^2+2(O(3)/R)\,dUdV+dV^2=0$, where
$$
P=V+\dfrac{\overline p_{20}}{2}U^2+O(3),\quad
R=
1+(\overline p_{02}-4 \overline q_{11}) U^2/4
-2 \overline q_{02} UV+O(3)
$$
and it is equal to $k=3$ of
\begin{equation}\label{eq:bde3}
\Big(V+\dfrac{\overline p_{20}}{2}U^2+O(k)\Big)
\,dU^2
+
2O(k)\,dUdV+dV^2=0.
\end{equation}
It known that 
for any $r\geq3$, the BDE
\eqref{eq:bde3}
is equivalent to $k=r$ of \eqref{eq:bde3}
(\cite[Proposition 4.9]{bt} see also \cite{d}).
We set
$$
A(\tilde\omega)=
\dfrac{p_{20}- 2q_{10}^2-p_{01}q_{10}}{p_{01}^2}
$$
for a BDE $\tilde\omega=0$ of the form 
\eqref{eq:bde1}.
Summarizing up the above arguments,
we have the following fact.
\begin{fact}
For any\/ $r\geq3$,
the BDE\/ $\tilde\omega$ of the form\/ \eqref{eq:bde1}
is equivalent to
$$
\left(u+
A(\tilde\omega)\dfrac{u^2}{2}
+O(r)\right)\,
du^2
+
2O(r)\,dudv+dv^2=0.
$$
\end{fact}
On the other hand,
the configuration of
the solutions of the BDE
$$
\omega_{l}
=(v+lu^2/2)\,du^2+dv^2=0
$$
is 
called {\it folded saddle\/} if $l<0$,
called {\it folded node\/} if $0<l<1/8$,
and 
called {\it folded focus\/} if $l>1/8$,
and they are drawn 
as in Figure \ref{fig:folded}.
\begin{figure}[!ht]
\centering
\includegraphics[width=.25\linewidth]{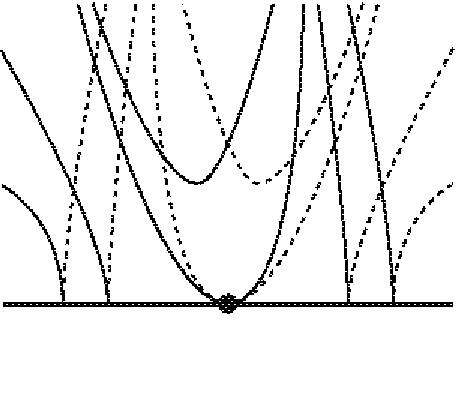}
\hspace{1mm}
\includegraphics[width=.25\linewidth]{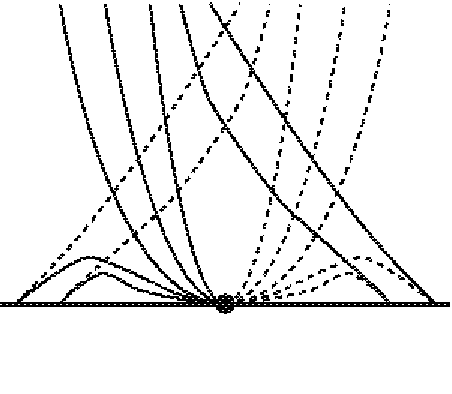}
\hspace{1mm}
\includegraphics[width=.25\linewidth]{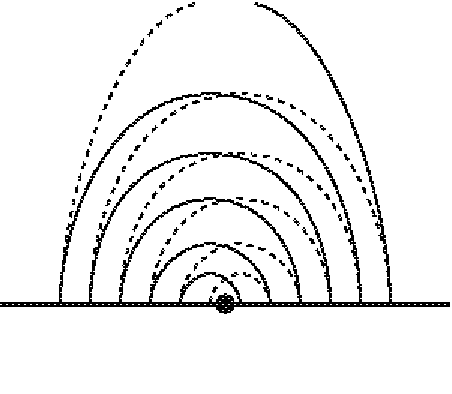}
\caption{Configurations of $\omega_k$,
folded saddle, folded node, folded focus.}
\label{fig:folded}
\end{figure}

\subsection{Geometric foliations near swallowtail}
Here we consider the following three
$2$-tensors.
\begin{equation}\label{eq:bdes}
\begin{array}{rl}
\omega_{lc}
=&
(F N - G M)\,du^2+
(E N - G L)\,dudv+
(E N - F L)\,dv^2,\\
\omega_{as}
=&
L\,du^2+
2M\,dudv+
N\,dv^2,\\
\omega_{ch}
=&
(L (G L - E N) + 2 M (EM-F L))\,du^2\\
&\hspace{10mm}
+
2(M (G L + E N) - 2 FLN)\,dudv\\
&\hspace{20mm}
+(N (EN-G L)+2 M (G M - F N) )\,dv^2.
\end{array}
\end{equation}
The configuration of the solutions of $\omega_{lc}$ is called the
{\it lines of curvature},
and that of $\omega_{as}$ is called the
{\it asymptotic curves}.
Since 
$\omega_{ch}=0$ 
can be deformed to
$$
\begin{array}{cl}
&(NH-GK)\,dv^2+2(MH-FK)\,dudv+(LH-EK)\,du^2=0\\[2mm]
\Leftrightarrow&
\dfrac{N\,dv^2+2M\,dudv+L\,du^2}
{G\,dv^2+2F\,dudv+E\,du^2}=\dfrac{K}{H}
\left(=
\dfrac{2}{\kappa_1^{-1}+\kappa_2^{-1}}\right),
\end{array}
$$
along the solution curves of $\omega_{ch}$,
its normal curvature is equal to the harmonic mean
of principal curvatures,
where $K,H$ are the Gaussian and mean curvatures
respectively.
The discriminant of $\omega_{ch}=0$ is a positive multiplication
of $K$.
Thus the solutions of $\omega_{ch}=0$ lie
in the region of positive Gaussian curvature.
The the solutions of $\omega_{ch}=0$ is
called {\it characteristic curves} (see \cite{e,gs}).
We consider three foliations of \eqref{eq:bdes}
near swallowtail.
Configurations of these foliations near
singular points are intensively studied.
See \cite{dara,d,ggs,gs,ttohoku,t}, for example.
Let $\nu_2$ be a normal vector to $f$ where we do not
assume $|\nu_2|=1$,
and set 
$L_2=\inner{f_{uu}}{\nu_2}$, 
$M_2=\inner{f_{uv}}{\nu_2}$, 
$N_2=\inner{f_{vv}}{\nu_2}$.
One can easily see that
all BDEs of
\eqref{eq:bdes}
are equivalent to
that of changing $L,M,N$ to $L_2,M_2,N_2$.

We take the coordinate system as 
in subsection \ref{sec:singuaxis}.
Then the singular set is $\{v=0\}$ and
the null vector field on the $u$-axis
is $\partial_u+u\partial_v$.
Thus $F_v(u,0)=0$ for any $u$.
Hence there exists a vector valued function
$\phi$ such that $F_u(u,v)+uF_v(u,v)=v\phi(u,v)$.
We set
$$
\tilde E_2=\inner{\phi}{\phi},\ 
\tilde F_2=\inner{\phi}{f_v},\ 
\tilde G_2=\inner{f_v}{f_v},$$
and
$$
\tilde L_2=-\inner{\phi}{(\nu_2)_u},\ 
\tilde M_2=-\inner{\phi}{(\nu_2)_v},\ 
\tilde N_2=-\inner{f_v}{(\nu_2)_v}.
$$
Then
\begin{align}
\omega_{lc}
=&
v\Bigg(
\Big(
 \tilde F \tilde N_2
-\tilde G \tilde M_2
\Big)\,du^2
+
\Big(
-\tilde G \tilde L_2
+\big(\tilde G \tilde M_2
-2 \tilde F \tilde N_2\big) u
+\tilde E \tilde N_2 v
\Big)\,
dudv\nonumber\\
&\hspace{5mm}
+
\Big(
 \tilde G \tilde L_2 u
-\tilde F \tilde L_2 v
+\tilde F \tilde N_2 u^2
+\big(-\tilde F \tilde M_2
-\tilde E \tilde N_2\big) u v
+\tilde E \tilde M_2 v^2
\Big)\,
dv^2\Bigg),
\end{align}
\begin{align}
\omega_{as}
=&
\Big(
 \tilde L_2v
+\tilde N_2 u^2
-\tilde M_2 uv
\Big)\,
du^2+
2\Big(-\tilde N_2 u+\tilde M_2 v\Big)\,
dudv
+
\tilde N_2\,
dv^2,
\end{align}
\begin{align}
\omega_{ch}
=&
v\Bigg(
\Big(
 \tilde G \tilde L_2^2 v
-\tilde G \tilde L_2 \tilde N_2 u^2
+4 \tilde F \tilde L_2 \tilde N_2 u v
-\big(2 \tilde F \tilde L_2 \tilde M_2
+\tilde E \tilde L_2 \tilde N_2\big) v^2
\nonumber\\
&\hspace{15mm}
-\tilde G \tilde M_2 \tilde N_2 u^3
+\big(\tilde G \tilde M_2^2
+2 \tilde F \tilde M_2 \tilde N_2
+\tilde E \tilde N_2^2\big) u^2 v
\nonumber\\
&\hspace{25mm}
-\big(2 \tilde F \tilde M_2^2
+3 \tilde E \tilde M_2 \tilde N_2\big) u v^2
+2 \tilde E \tilde M_2^2 v^3
\Big)\,du^2\nonumber\\
&\hspace{5mm}
+
2\Big(
 \tilde G \tilde L_2 \tilde N_2 u
+\big(\tilde G \tilde L_2 \tilde M_2
-2 \tilde F \tilde L_2 \tilde N_2 \big)v\nonumber\\
&\hspace{15mm}
+\tilde G \tilde M_2 \tilde N_2 u^2
-\big(\tilde G \tilde M_2^2
+\tilde E \tilde N_2^2\big) u v
+\tilde E \tilde M_2 \tilde N_2 v^2
\Big)
\,dudv\nonumber\\
&
+
\Big(
-\tilde G \tilde L_2 \tilde N_2
-\tilde G \tilde M_2 \tilde N_2 u
+\big(2 \tilde G \tilde M_2^2 
-2 \tilde F \tilde M_2 \tilde N_2 
+\tilde E \tilde N_2^2 \big)v
\Big)\,
dv^2\Bigg).
\end{align}
We factor out it from $\omega_{lc}$ and $\omega_{ch}$
and set
$$
\tilde\omega_{lc}=\omega_{lc}/v,\quad
\tilde\omega_{as}=\omega_{as},\quad
\tilde\omega_{ch}=\omega_{ch}/v.
$$
We consider the solutions of 
$\tilde\omega_{lc}$, $\tilde\omega_{as}$ and
$\tilde\omega_{ch}$
instead of 
$\omega_{lc}$, $\omega_{as}$ and $\omega_{ch}$.
Let us consider $\tilde \omega_{lc}=0$.
Then the discriminant $\delta$ of $\tilde \omega_{lc}=0$
satisfies $\delta(0)>0$.
It is known that
such BDE is equivalent to $dx^2-dy^2=0$,
and its configuration is a pair of
transverse smooth foliations.
Existence of lines of curvature coordinate system
near swallowtail is shown by \cite{mu}.
Let us consider $\tilde \omega_{as}=0$ and
$\tilde \omega_{ch}=0$.
We set $\tilde \omega_{as}=
p_{as}\,du^2+2q_{as}\,dudv+r_{as}\,dv^2$ and
$\tilde \omega_{ch}=
p_{ch}\,du^2+2q_{ch}\,dudv+r_{ch}\,dv^2$.
We assume that $\kappa_\nu(0)=-2 b_{12} + b_{20}\ne0$.
Then $r_{as}$ and $r_{ch}$ does not vanish at $0$.
Thus 
$\tilde \omega_{as}$ (respectively, $\tilde \omega_{ch}$)
is equivalent to
$$
\bar\omega_{as}=\dfrac{p_{as}}{r_{as}}\,du^2
+2\dfrac{q_{as}}{r_{as}}\,dudv+dv^2\quad
\Big(
\text{respectively, }
\bar \omega_{ch}=\dfrac{p_{ch}}{r_{ch}}\,du^2
+2\dfrac{q_{ch}}{r_{ch}}\,dudv+dv^2
\Big).
$$
We see that
$$
\dfrac{p_{as}}{r_{as}}
=
-\dfrac{b_{04}}{2 b_{12}-b_{20}}v+u^2
+*u v+*v^2+O(3),\quad
\dfrac{q_{as}}{r_{as}}
=
-u+*v+O(2)
$$
and
$$
\dfrac{p_{as}}{r_{as}}
=
\dfrac{b_{04}}{2 b_{12}-b_{20}}v+u^2
+*u v+*v^2+O(3),\quad
\dfrac{q_{as}}{r_{as}}
=
-u+*v+O(2).
$$
By \eqref{eq:invkeisan}, we have 
$$
A(\tilde\omega_{as})
=
-\dfrac{b_{04}}{2 b_{12}-b_{20}}
=
-\dfrac{\mu_c\tau_s}{4\kappa_\nu}
\quad
\left(
\text{respectively, }
A(\tilde\omega_{ch})
=
\dfrac{b_{04}}{2 b_{12}-b_{20}}
=
\dfrac{\mu_c\tau_s}{4\kappa_\nu}
\right).
$$
Thus we know the configuration of the solutions
of $\tilde\omega_{as}$ is
folded saddle if 
$l=-\mu_c\tau_s/(4\kappa_\nu)$ $<0$,
folded node if 
$0<l<1/8$, and
folded focus if 
$1/8<l$.
The same holds for $\tilde\omega_{ch}$ by
setting
$l=-\mu_c\tau_s/(4\kappa_\nu)<0$.
\begin{figure}[!ht]
\centering
\begin{tabular}{ccccc}
\includegraphics[width=.25\linewidth]{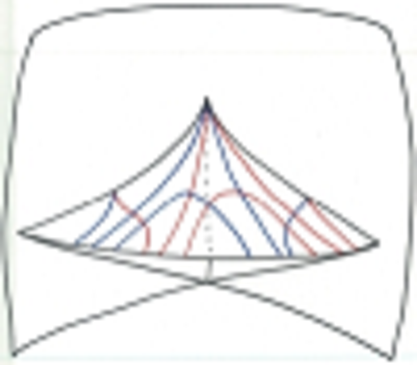}
\hspace{1mm}
\includegraphics[width=.25\linewidth]{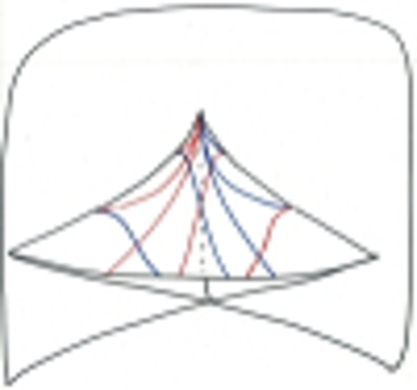}
\hspace{1mm}
\includegraphics[width=.25\linewidth]{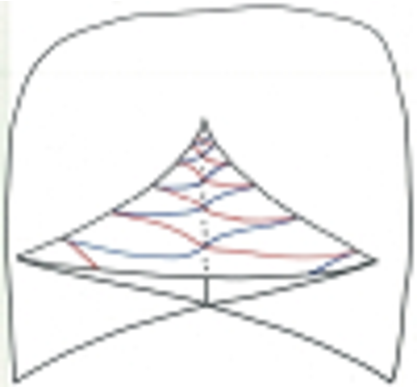}\\
\includegraphics[width=.25\linewidth]{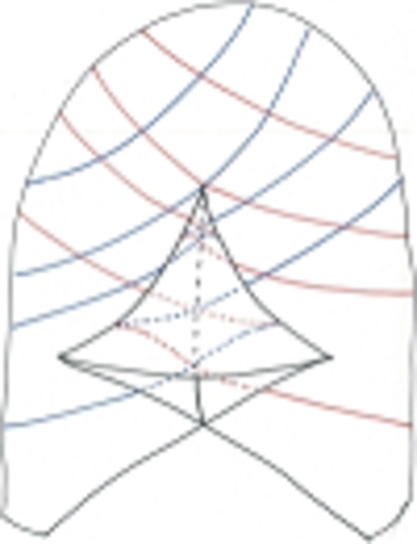}
\hspace{1mm}
\includegraphics[width=.25\linewidth]{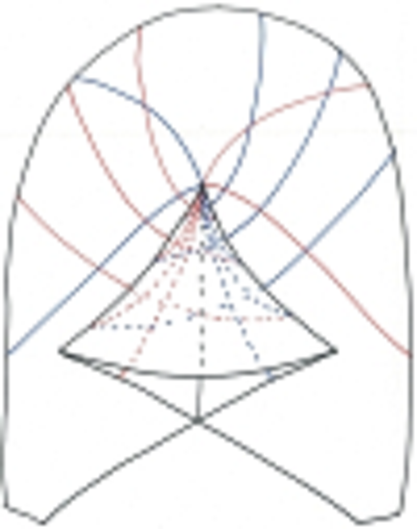}
\hspace{1mm}
\includegraphics[width=.25\linewidth]{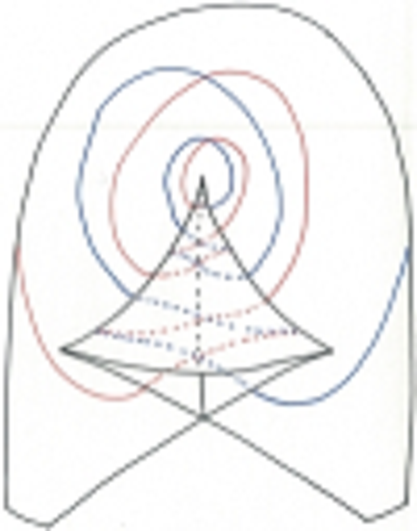}
\end{tabular}
\caption{Geometric foliations on swallowtail}
\label{fig:folsw}
\end{figure}
We can draw models of $\tilde\omega_{as}$ and $\tilde\omega_{ch}$
at swallowtails as in Figure \ref{fig:folsw}.
Since swallowtails appear as points on fronts,
we would like to say that
the generic configuration of $\tilde\omega_{as}$ and
$\tilde\omega_{ch}$ are the these types.
We note that since
$\tilde\omega_{as}|_{v=0}
=\tilde{N}_2(u^2\,du^2-2u\,dudv+dv^2)$,
$\tilde\omega_{ch}|_{v=0}
=-(\tilde{N}\tilde{L}_2\tilde{N}_2(u^2\,du^2-2u\,dudv+dv^2)$,
the direction of $\tilde\omega_{as}$ and $\tilde\omega_{ch}$
are not the direction of the null vector field,
the solutions of them on $S(f)$ forms
$3/2$-cusps on the image of $f$.

\begin{flushright}
\begin{tabular}{l}
Department of Mathematics, \\
Kobe University,
Rokko 1-1, Nada, \\
Kobe 657-8501, Japan\\
{\tt sajiO\!\!\!amath.kobe-u.ac.jp}\\
\end{tabular}
\end{flushright}

\end{document}